\documentclass[12pt]{amsart}
\usepackage{graphicx}
\usepackage{epstopdf}
\usepackage{amsmath,amsthm,amssymb,amscd}
\usepackage{ascmac}
\usepackage[arrow,matrix]{xy}
\usepackage{enumerate}
\usepackage{mathrsfs}

\newtheorem{thm}{Theorem}[section]
\newtheorem{dfn}[thm]{Definition}
\newtheorem{exa}[thm]{Example}

\newtheorem{lem}[thm]{Lemma}

\newcommand{\R}{\mathbb R}

\title
[Ribbon-clasp surface-links and normal forms]
{Ribbon-clasp surface-links and normal forms of singular surface-links}

\author[S.~Kamada]{Seiichi KAMADA}
\address[S.~Kamada]{Department of Mathematics, Osaka City University, Sugimoto 3-3-138, Sumiyoshi-ku, Osaka 558-8585, Japan}
\email{skamada@sci.osaka-cu.ac.jp}

\author[K.~Kawamura]{Kengo KAWAMURA}
\address[K.~Kawamura]{Department of Mathematics, Osaka City University, Sugimoto 3-3-138, Sumiyoshi-ku, Osaka 558-8585, Japan}
\email{k.kawamura0403@gmail.com}



\begin{document}

\maketitle

\begin{abstract}
We introduce the notion of a ribbon-clasp surface-link, which is a generalization of a ribbon surface-link.  We generalize the notion of a normal form on embedded surface-links to the case of immersed surface-links and prove that any (immersed) surface-link can be described in a normal form.  It is known that an embedded  surface-link is a ribbon surface-link if and only if it can be described in a symmetric normal form.  We prove that 
an (immersed) surface-link is a ribbon-clasp surface-link if and only if it can be described in a symmetric normal form.  We also introduce the notion of a ribbon-clasp normal form, which is a simpler version of a symmetric normal form.  
\end{abstract}


\section{Introduction}

In this paper an {\it immerse surface-knot} or simply a {\it surface-knot} 
means a closed connected and oriented surface 
generically immersed in $\R^4$. When it is embedded, we also call it an {\it embedded surface-knot}.  A {\it surface-link} is a disjoint union of surface-knots. 
Two surface-links are said to be {\it equivalent} if they are ambient isotopic. 

A surface-link is called {\it trivial} or {\it unknotted} if it is the boundary of a disjoint union of embedded handlebodies in $\R^4$.  A surface-link is called {\it ribbon} if it is the boundary of immersed handlebodies in $\R^4$ whose multiple point set is a union of ribbon singularities.  (The definition of a ribbon singularity is given in Section~\ref{sect:singularity}.  For an immersion $f : M \to \R^4$ of a compact $3$-manifold $M$,  the {\it boundary} of the immersed $3$-manifold $f(M)$ means the image $f(\partial M)$ of the boundary $\partial M$ of $M$.)   

\begin{dfn}{\rm 
A surface-link is {\it ribbon-clasp} if it is the boundary of immersed handlebodies in $\R^4$ whose multiple point set is a union of ribbon singularities and clasp singularities. 
}\end{dfn} 

The definition of a clasp singularity is given in Section~\ref{sect:singularity}. 
By definition, a trivial surface-link is ribbon, and a ribbon surface-link is ribbon-clasp.  

Let $F$ be a ribbon-clasp surface-link and let $f: M \to \R^4$ be an immersion of handlebodies $M$ with $F= f(\partial M)$ such that the multiple point set is a union of ribbon singularities and $c$ $(\geq 0)$ clasp singularities.  There are two double  points of $F$ in each clasp singularity, and one of them is a positive double point and the other is negative.  Thus $F$ has $c$ positive double points and $c$ negative double points, and the normal Euler number of $F$ is zero.  

\begin{thm} \label{thm:AA}
For a surface-link $F$, the following conditions are equivalent. 
\begin{itemize}
\item[(1)] $F$ is a ribbon-clasp surface-link.  
\item[(2)] $F$ is obtained from a ribbon surface-link by finger moves.  
\item[(3)] $F$ is obtained from a trivial $2$-link by $1$-handle surgeries and finger moves. 
\item[(4)] $F$ is obtained from an M-trivial $2$-link by $1$-handle surgeries. 
\end{itemize} 
\end{thm}

$1$-handle surgeries, finger moves and M-trivial $2$-links are explained in Section~\ref{sect:singularity} and Theorem~\ref{thm:AA} is proved in Section~\ref{sect:proofAA}.   

In Section~\ref{sect:normalform} we define a normal form for a ribbon-clasp surface-link called a {\it ribbon-clasp normal form} 
 in terms of the motion picture method.  We prove the following theorem in Section~\ref{sect:proofAB}. 

\begin{thm} \label{thm:AB}
A surface-link is ribbon-clasp if and only if it is equivalent to a surface-link in a ribbon-clasp normal form.  
\end{thm}

In Section~\ref{sect:normalform} we also define a normal form for a (singular) surface-link.  It is a generalization of a normal form for an embedded surface-link 
defined in Kawauchi-Shibuya-Suzuki \cite{KSS1}.  

\begin{thm} \label{thm:B}
Any surface-link $F$ is equivalent to a surface-link $F'$ in a normal form.  Moreover, 
for any non-negative integers $c'_+$, $c'_-$, $c''_+$ and $c''_-$ such that $c'_+ + c''_+$ is the number of positive double points of $F$ and $c'_- + c''_-$ is the number of negative double points of $F$, we may assume that $F'$ has $c'_+$ positive double points and $c'_-$ negative double points in $\R^3 \times [0, \infty)$ and $c''_+$ positive double points and $c''_-$ negative double points in $\R^3 \times (-\infty, 0]$.  
\end{thm}

When $F$ is an embedded surface-link, Theorem~\ref{thm:B} 
 is the   theorem on normal forms given in \cite{KSS1} (cf. \cite{Kam2002}).  

We say that a surface-link is in a {\it symmetric normal form} if it is in a normal form and it is symmetric with respect to the hyperplane $\R^3 \times \{0\}$.  

\begin{thm} \label{thm:C}
A surface-link is ribbon-clasp if and only if 
it is equivalent to a surface-link in a symmetric normal form.
\end{thm}

When $F$ is an embedded surface-link, the above theorem is the characterization of a ribbon surface-link due to Kawauchi-Shibuya-Suzuki \cite{KSS2}.  

\section{Preliminaries}\label{sect:singularity}

In this section, we give definitions of a ribbon singularity, a clasp singularity, a $1$-handle surgery, a finger move and an M-trivial $2$-link.  

Let $M$ be a compact $3$-manifold with non-empty boundary and $f:M \rightarrow \mathbb{R}^4$ an immersion of $M$ into $\mathbb{R}^4$. 
Let $\Delta$ be a connected component of  the  multiple point set $\{x \in f(M) \, \mid \, \#f^{-1}(x) \geq 2\}$ $\subset \R^4$. 

We say that $\Delta$ is a {\it ribbon singularity} if $\Delta$ is a $2$-disk in $\mathbb{R}^4$ and the preimage of $\Delta$ is the disjoint union of embedded $2$-disks $\Delta_1$ and $\Delta_2$ in $M$ such that $\Delta_1$ 
 is properly embedded in $M$ and $\Delta_2$  is embedded in the interior of $M$. 
Figure \ref{fig:ribbon1} shows a local model of a ribbon singularity and a motion picture.  

\begin{figure}[h]
 \centering
 \includegraphics{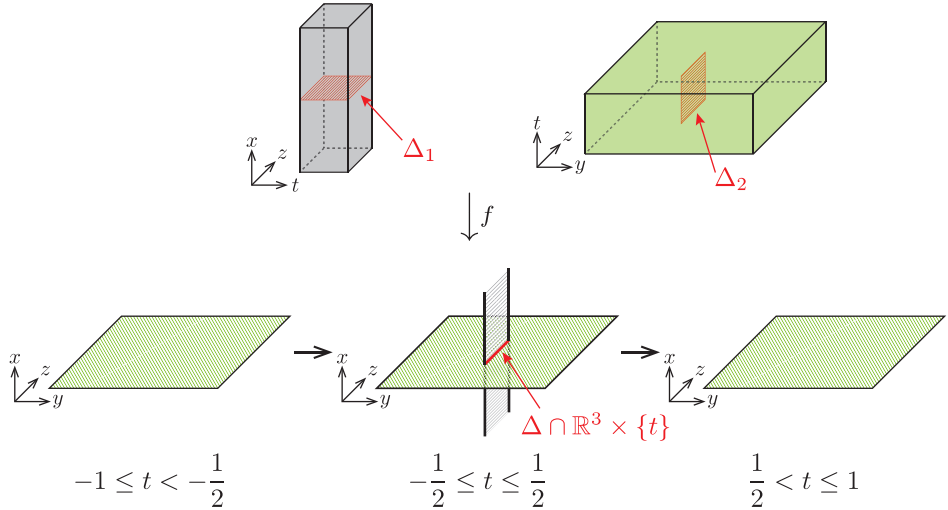}
 \caption{A local model of a ribbon singularity}
 \label{fig:ribbon1}
\end{figure}

We say that $\Delta$ is a {\it clasp singularity} if $\Delta$ is a $2$-disk in $\mathbb{R}^4$ and the preimage of $\Delta$ is the disjoint union of embedded $2$-disks $\Delta_1$ and $\Delta_2$ in $M$ such that for each $ i \in \{1,2\}$,  $\partial\Delta_i$ is the union of two arcs $\alpha_i$ and $\beta_i$, where $\alpha_i$ is a properly embedded arc in $M$ and $\beta_i$ is a simple arc in $\partial M$ which connects endpoints of $\alpha_i$.
Figure \ref{fig:clasp1} shows a local model of a clasp singularity and a motion picture. 

\begin{figure}[h]
 \centering
 \includegraphics{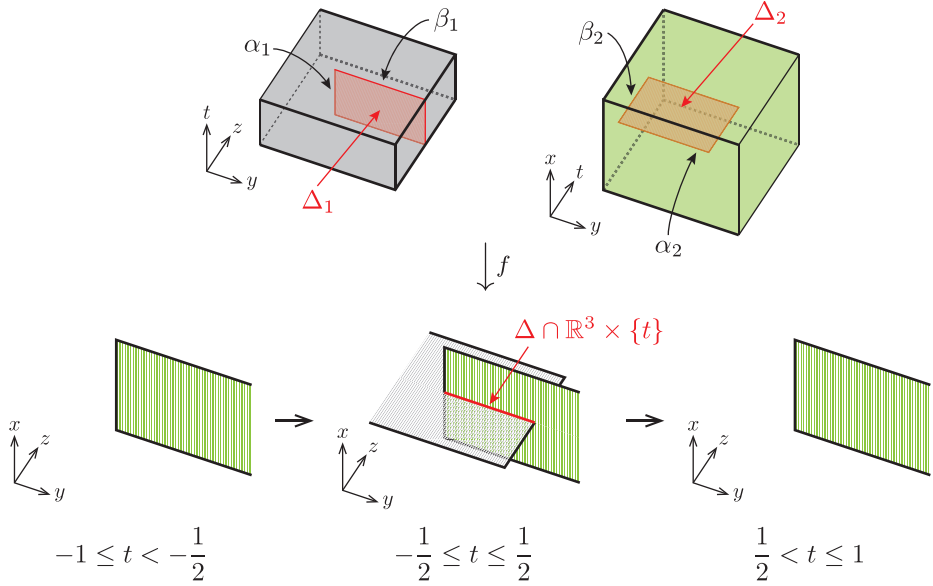}
 \caption{A local model of a clasp singularity}
 \label{fig:clasp1}
\end{figure}

\vspace{0.3cm} 

Let $F$ be a surface-link. 
A {\it chord} attaching to $F$ means an unoriented simple arc $\gamma$ in $\R^4$ 
such that $F \cap \gamma= \partial \gamma$, which  misses the double points of $F$.  
Two chords are {\it equivalent} or {\it homotopic} with respect to $F$ if they are homotopic through chords attaching to $F$.

A {\it $1$-handle} attaching to $F$ means an embedded $3$-disk $B$ in $\R^4$ 
such that $F \cap B$ is the union of a pair of mutually disjoint $2$-disks in $\partial B$, $F \cap B$ misses the double points of $F$, 
and the orientation of $F \cap B$ induced from $\partial B$ is opposite to the orientation induced from $F$.  
Put 
$${\rm h}(F; B):= 
{\rm Cl}(F  \cup \partial B - F \cap B),$$ which we call the  surface-link 
obtained from $F$ by a {\it $1$-handle surgery} along $B$. 
(Here ${\rm Cl}$ means the closure.) 
Two $1$-handles attaching to $F$ are said to be {\it equivalent} if they are ambient isotopic by an isotopy of $\R^4$ keeping $F$ setwise fixed.  
It is known that 
$1$-handles attaching to $F$ are equivalent if and only if their cores are equivalent as chords attaching to $F$.  (Boyle \cite{Bo} and Hosokawa-Kawauchi \cite{HK}).  
For a chord $\gamma$ attaching to $F$, we  denote by ${\rm h}(F; \gamma)$ the surface-link obtained from $F$ by a $1$-handle surgery along a $1$-handle whose core is $\gamma$.  

A finger move is  the inverse operation of the Whitney trick (cf. \cite{C, Kirb}).  
Here we use an alternative definition using a Montesinos twin and $1$-handle surgeries.  A {\it Montesinos twin} is a surface-link   which is the boundary of 
a pair of embedded oriented $3$-disks $B^3_1$ and $B^3_2$ with a single ribbon singularity between $B^3_1$ and $B^3_2$.   It has a positive double point and a negative double point.  The equivalence class, as a surface-link,  of a Montesinos twin is unique.  We denote by $T$ a Montesinos twin.  

Let $F$ be a surface-link.  Let $U$ be a $4$-disk in $\R^4$ disjoint from $F$.  Put a Montesinos twin $T= S_1 \cup S_2$ in $U$ and let $B^3_1$ and $B^3_2$ be embedded $3$-disks in $U$ with a clasp singularity with $S_i = \partial B^3_i$ $(i \in \{1,2\})$.  
Take a point $p_1\in S_1 - S_1\cap S_2$, a point $p_2 \in S_2 - S_1 \cap S_2$ and two distinct points $q_1, q_2$ of $F$ missing the double points of $F$.  Let $\gamma_1$ and $\gamma_2$ be oriented chords attaching to $F \cup M$ such that 
for $i \in \{1,2\}$,  
$\gamma_i$ starts from $p_i$ and terminates at $q_i$ and 
$\gamma_i \cap (B^3_1 \cup B^3_2)  = \{p_i\}$.   
Let $F'$ be the surface-link obtained from $F \cup T$ by $1$-handle surgeries along two $1$-handles whose cores are $\gamma_1$ and $\gamma_2$.  
Let $\gamma$ be a chord attaching to $F$ which is a concatenation of $\gamma_1^{-1}$, a simple arc from $p_1$ to $p_2$ in $U$ and $\gamma_2$.  
We say that $F'$ is obtained from $F$ by a {\it finger move} along $\gamma$, which we denote by ${\rm f}(F; \gamma)$.   
It is seen that if $\gamma$ and $\gamma'$ are equivalent chords attaching to $F$ then ${\rm f}(F; \gamma)$ is equivalent to ${\rm f}(F; \gamma')$.  

\begin{dfn}{\rm 
An {\it M-trivial $2$-link} is a surface-link which is a split union of a trivial $2$-link and some (or no) Montesinos twins.  
}\end{dfn} 

\section{Proof of Theorem~\ref{thm:AA}}\label{sect:proofAA}

\begin{proof}[Proof of Theorem~\ref{thm:AA}.]
Suppose (1).  
Let $F$ be a ribbon-clasp surface-link.  Let $f : M \to \R^4$ be an immersion of handlebodies $M$ with $F = f(\partial M)$ whose multiple point set is a union of ribbon singularities and clasp singularities. Let $\Delta_1, \dots, \Delta_c$ be the clasp singularities, and for each $i \in \{1, \dots, c\}$, let $\Delta_{i1} $ and $\Delta_{i2}$ be the $2$-disks in $M$ with $f^{-1}(\Delta_i) = 
\Delta_{i1} \cup \Delta_{i2}$.  Let $N_{ij}$ $(i \in \{1,\dots, c\}, j \in \{1,2\})$ be a regular neighborhood of $\Delta_{ij}$ in $M$ and let $D_{ij}$ be the  
properly embedded $2$-disk in $M$ with $D_{ij} = {\rm Cl}(\partial N_{ij} \cap {\rm Int} M)$.  Let $D_{ij} \times [0,1]$ be a regular neighborhood of $D_{ij}$ in 
${\rm Cl}(M - N_{ij})$. Put $M_0 = {\rm Cl}(M - \cup_{i,j} (N_{ij} \cup D_{ij}\times [0,1])$ where $i$ and $j$ run over all $i \in \{1, \dots, c\}$ and $j \in \{1,2\}$.  Let $f_0: M_0 \to \R^4$ be the restriction of $f$ to $M_0$. Since $M_0$ is a $3$-manifold homeomorphic to $M$, which is a union of handlebodies.  The multiple point set of the immersion $f_0$ is 
a union of ribbon singularities, and $F_0:= f_0(\partial M_0)$ is a ribbon surface-link.  
For each $i \in \{1, \dots, c\}$, let $T_i := f( \partial N_{i1} \cup \partial N_{i2})$, which is a Montesinos twin.  The image $f(D_{ij}\times [0,1])$ is a $1$-handle attaching to $F_0 \cup 
(\cup_i T_i)$, and 
we see that $F$ is obtained from 
$F_0$ by finger moves using 
$T_1, \dots, T_c$ and $1$-handles $f(D_{ij}\times [0,1])$.  Thus we have (2).  

Suppose (2).  
Let $F$ be obtained from a ribbon surface-link $F_0$ by finger moves.  
The ribbon surface-link $F_0$ is obtained from a trivial $2$-link, say $F_{00}$, by $1$-handle surgeries.  More precisely, there are mutually disjoint $1$-handles, say $B_1, \dots, B_n$, attaching to $F_{00}$ and the surgery result is $F_0$. 
If necessary modifying the finger moves, we may assume that the chords attaching to $F_0$ along which the finger moves are applied are disjoint from the $1$-handles $B_1, \dots, B_n$.  Then we can apply the $1$-handle surgeries along $B_1, \dots, B_n$ and the finger moves simultaneously to $F_{00}$.  Thus we have (3).   Since $F_{00} \cup (\cup_i T_i)$ is an $M$-trivial $2$-link, we also have (4).  

Suppose (3).  
Let $F$ be obtained from a trivial $2$-link by $1$-handle surgeries and finger moves.  Without loss of generality, we may assume that $F$ is obtained from a trivial $2$-link by $1$-handle surgeries first and then by finger moves, namely, $F$ is obtained from a ribbon surface-link $F_0$ by finger moves.  Thus we have (2).  

Suppose (4).  
Let $F$ be obtained from an M-trivial $2$-link $F'$ by $1$-handle surgeries.  
There is an immersion $f': M' \to \R^4$ with $f'(\partial M')= F'$ such that $M'$ is a disjoint union of $3$-disks and  
the multiple point set of $f'$ is a union of clasp singularities.  
Modifying the cores of the $1$-handles attaching to $F'$ used to obtain $F$, we assume that 
the cores intersect with $f'(M')$ in general position.  Then we have an immersion $f: M \to \R^4$ with $f(\partial M)=F$ 
which is an extension of $f'$ 
such that $M$ is a disjoint union of handlebodies and the multiple point set of $f$ is a union of clasp singularities and ribbon singularities.  
Thus we have (1).  
This completes the proof of Theorem~\ref{thm:AA}.  
\end{proof}

\section{Normal Forms}\label{sect:normalform}

First we prepare some terminologies.  We allow a link to be empty, which we call the  empty link and denote it by $\emptyset$. 

Let $L$ be a link and let $\mathcal{D} =\{D_1, \dots, D_m\}$  be a set of 
mutually disjoint $m$ $(\geq 1)$ oriented $2$-disks  in $\R^3$ with $L \cap D_i = \emptyset$ for $i \in \{1, \dots, m\}$.  We say that the link $L \cup ( \cup_i  \partial D_i)$ is obtained from $L$ by a {\it $0$-surgery} along $\mathcal{D}$.  We also say that $L$ is obtained from $L \cup ( \cup_i  \partial D_i)$ by a {\it $2$-surgery}.  

Let $L$ be a link and let $\mathcal{D} =\{D_1, \dots, D_m\}$ be a set of 
mutually disjoint oriented $2$-disks  in $\R^3$ such that for $i \in \{1, \dots, m\}$, 
$D_i$ is a band attaching to $L$. 
We say that the link 
$${\rm h}(L; \mathcal{D}):= 
{\rm Cl}(L \cup (\cup_i \partial D_i) -  L \cap (\cup_i \partial D_i))$$  
 is obtained from $L$ by a {\it $1$-surgery}.  
A $1$-surgery is also called {\it surgery} or {\it band surgeries}.  
(A {\it band} attaching to $L$ is an oriented $2$-disk $D$ in $\R^3$ such that $L \cap D$ is the union of a pair of mutually disjoint arcs in $\partial D$ and the orientation of $L \cap \partial D$ induced from $\partial D$ is opposite to the orientation induced from $L$.)

Let $L$ be a link and let $(h_s \mid s \in [0,1])$ be an ambient isotopy of $\R^3$, i.e., 
each $h_s: \R^3 \to \R^3$ is a homeomorphism and $h_0 = {\rm id}$.  When $L'=h_1(L)$, we say that $L'$ is obtained from $L$ by an {\it isotopic deformation}.  When each $h_s$ is the identity map, we call the isotopic deformation the {\it identity deformation}. 

Let $L$ be a link and let $L'$ be obtained from $L$ by applying some crossing changes. 
There is a homotopy $(g_s:  M^1 \to \R^3 \mid s \in [0,1])$ of the source circles $M^1$ of the link into $\R^3$ with 
$g_0(M^1) = L$ and $g_1(M^1) = L'$ such that each $g_s$, except $s= 1/2$, is an embedding of $M^1$ 
and 
at $s=1/2$ intersections occur.  We call such a homotopy a {\it crossing  change deformation}.  It is  called 
a {\it simple crossing change deformation} if exactly one crossing change occurs.   

A {\it link transformation} from $L$ to $L'$, denoted by $L \to L'$, means one of the following transformation changing   $L$ to $L'$: 
\begin{itemize}

\item[(1)] an isotopic deformation by an ambient isotopy 
$(h_s \mid s \in [0,1])$ of $\R^3$,    

\item[(2)] a $k$-surgery along $\mathcal{D}$ for some $k \in \{0,1,2\}$, 

\item[(3)] a crossing change deformation by a homotopy $(g_s \mid s \in [0,1])$. 
\end{itemize}
For an interval $[a,b]$ $(a < b)$, the {\it realizing surface} of $L \to L'$ in $\R^3 \times [a,b]$ is a compact oriented surface, say $F$, such that 
\begin{itemize}

\item[(1)] for an isotopic deformation by an ambient isotopy 
$(h_s \mid s \in [0,1])$ of $\R^3$, 
$F \cap \R^3 \times \{t\} = h_s(L) \times \{t\}$ for $t \in [a,b]$, where $s = (t-a)/(b-a)$,  

\item[(2)] for a $k$-surgery along $\mathcal{D} =\{D_1, \dots, D_m\}$ for some $k \in \{0,1,2\}$, 
$$F \cap \R^3 \times \{t\} = 
\left\{
\begin{array}{llc}
{\rm h}(L; \mathcal{D}) \times \{t\} & \mbox{for } m \in ( (a+b)/2, b]  &   \\
L \cup  ( \cup_i  D_i) \times \{t\} & \mbox{for } m= (a+b)/2  &   \\
L \times \{t\} & \mbox{for } m \in [a, (a+b)/2), &   
\end{array}
\right.
$$

\item[(3)] for a crossing change deformation by a homotopy $(g_s \mid s \in [0,1])$, 
$F \cap \R^3 \times \{t\} = g_s(L) \times \{t\}$ for $t \in [a,b]$, where $s = (t-a)/(b-a)$.  
\end{itemize}
We denote the realizing surface by $F(L\to L')_{[a,b]}$.  

A {\it link transformation sequence} is a sequence of link transformations 
$\mathcal{L}: L_1 \to L_2 \to \dots \to L_m$.  
A {\it division} of an interval $[a,b]$ of length $m$ is a 
sequence of real numbers 
$t_1, t_2, \dots, t_m$ with  $a=t_1 < t_2 < \dots < t_m=b$.   
The {\it realizing surface} of a link transformation sequence $\mathcal{L} : L_1 \to L_2 \to \dots \to L_m$ 
in $\R^3 \times [a,b]$ with a division $t_1, t_2, \dots, t_m$ is 
the union of the realizing surfaces $F(L_i \to L_{i+1})_{[t_i, t_{i+1}]}$ for $i=1, \dots, m-1$.  It is a surface in $\R^3 \times [a,b]$, which we denote by $F(\mathcal{L})_{[t_1, \dots, t_m]}$ or $F(\mathcal{L})_{[a,b]}$.  

\vspace{0.3cm} 
We call a link $L$ an {\it H-trivial link} with $m$ Hopf links if it is a split union of a trivial link and $m$ Hopf links for some $m \geq 0$.  

We say that a link transformation $L \to L'$ is  a 
{\it Hopf-splitting deformation} when $L \to L'$ is a crossing change deformation  such that each crossing change occurs for a Hopf  link to change it into a trivial $2$-component link.  A {\it simple Hopf-splitting deformation} is a Hopf-splitting deformation that changes exactly one Hopf  link to a trivial $2$-component link.  

Note that when $L \to L'$ is  a 
{\it Hopf-splitting deformation} from 
an H-trivial link $L$  with $m$ Hopf links to  a trivial link $L'$, 
 the number of crossing changes is $m$.   

Let $L \to L'$ be a link transformation that is a $1$-surgery along $\mathcal{D} =\{D_1, \dots, D_m\}$.  When $\# L' = \# L -m$, we call it a $1$-surgery of {\it fusion type}. 
When $\# L' = \# L +m$, we call it a $1$-surgery of {\it fission type}. 

\vspace{0.3cm}  

\begin{dfn}\label{dfn:ribbonclaspnormalform}{\rm 
A surface-link $F$ is in a {\it ribbon-clasp normal form} if it is the realizing surface of a link transformation sequence 
$$ \mathcal{L}: 
\emptyset \to O \to L \to L_0 \to L \to O \to \emptyset $$
in $\R^3 \times [-3,3]$ with a division $-3, -2, -1, 0, 1, 2, 3$ 
satisfying the following conditions. 
\begin{itemize} 
\item[(1)] 
$O$ is a trivial link and 
$O \to \emptyset$ is a $2$-surgery.  
\item[(2)] 
$L$ is an H-trivial link and 
$L \to O$ is a Hopf-splitting deformation. 
\item[(3)] 
$L_0$ is a link and 
$L_0 \to L$ is a $1$-surgery. 
\item[(4)]  
The realizing surface is symmetric with respect to $\R^3 \times \{0\}$.  
\end{itemize}
}\end{dfn} 

\begin{exa}{\rm 
Figure~\ref{fig:fg_st2sumB2} shows a surface-knot in a ribbon-clasp normal form. 
}\end{exa}

\begin{figure}[h]
 \centering
 \includegraphics[width=12.5cm]{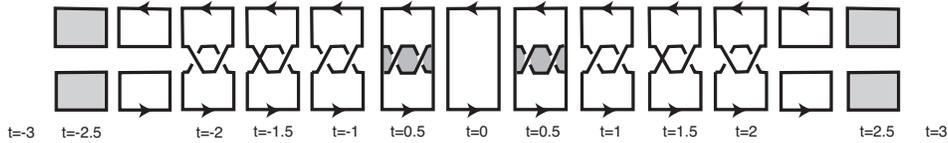}
 \caption{A surface-link in a ribbon-clasp normal form}
 \label{fig:fg_st2sumB2}
\end{figure}

The following is the definition of a {\it normal form} for a (singular) surface-link.  

\begin{dfn}{\rm 
Let $F$ be a surface-link with $\mu$ components and $g$ total genus.  We say that $F$ is in a {\it normal form} if it is the realizing surface of a link transformation sequence 
$$ \mathcal{L}: 
\emptyset \to O_- \to L_- \to K_- \to L_0 \to K_+ \to L_+ \to O_+ \to \emptyset $$
in $\R^3 \times [-4,4]$ with a division $-4, -3, \cdots, 3, 4$ 
satisfying the following conditions. 
\begin{itemize} 
\item[(1)] $O_-$ and $O_+$ are trivial links, 
$\emptyset \to O_-$ is a $0$-surgery and 
$O_+ \to \emptyset$ is a $2$-surgery.  
\item[(2)] 
$L_-$ and $L_+$ are H-trivial links, 
$O_- \to L_-$ is the inverse of a Hopf-splitting deformation, and 
$L_+ \to O_+$ is a Hopf-splitting deformation. 
\item[(3)] 
$K_-$ and $K_+$ are links with $\mu$ components, 
$L_- \to K_-$ is a $1$-surgery of fusion type and 
$K_+ \to L_+$ is a $1$-surgery of fission type. 
\item[(4)] 
$L_0$ is a link with $\mu + g$ components, 
$K_- \to L_0$ is a $1$-surgery of fission type and 
$L_0 \to K_+$ is a $1$-surgery of fusion type.  
\item[(5)] Let $\mathcal{D}^{--}$ ($\mathcal{D}^{-}$, $\mathcal{D}^{+}$ or $\mathcal{D}^{++}$, resp.) be the set of bands used for the $1$-surgery 
$L_- \to K_-$ ($K_- \to L_0$, $L_0 \to K_+$ or $K_+ \to L_+$, resp.). 
Then any band in $\mathcal{D}^{--}$ is disjoint from any band in  $\mathcal{D}^{-}$, and any band in $\mathcal{D}^{++}$ is disjoint from any band in  $\mathcal{D}^{+}$. 
\end{itemize}
}\end{dfn} 

Let $F$ be in a normal form.  Let $c'_+$, $c'_-$, $c''_+$  and $c''_-$ be non-negative integers such that $F$ has $c'_+$ positive double points and $c'_-$ negative double points in $\R^3 \times [0, \infty)$ and $c''_+$ positive double points and $c''_-$ negative double points in $\R^3 \times (-\infty, 0]$.  Then $c'_+$ is the number of negative Hopf links in $L_+$,  $c'_-$ is the number of positive Hopf links in $L_+$, $c''_+$ is the number of positive Hopf links in $L_-$  and $c''_-$ is the number of negative Hopf links in $L_-$.  

\begin{exa}{\rm 
The surface-knot illustrated in Figure~\ref{fig:fg_st2knot_negA2} is an immersed $2$-sphere called a  
{\it negative standard singular $2$-knot}, which we denote by $S_-$.  It has a negative double point. 

\begin{figure}[h]
 \centering
 \includegraphics[width=8cm]{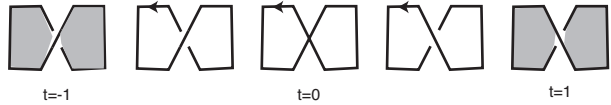}
 \caption{A negative standard singular $2$-knot}
 \label{fig:fg_st2knot_negA2}
\end{figure}

(1) When we take $c'_-=1$ and $c'_+ = c''_+ = c''_-=0$, we have a surface-knot  in a normal form illustrated 
in Figure~\ref{fig:fg_st2knot_negB2} which is equivalent to $S_-$.  
It is a realizing surface of  a link transformation sequence 
$$ \mathcal{L}: 
\emptyset \to O_- \to L_- \to K_- \to L_0 \to K_+ \to L_+ \to O_+ \to \emptyset $$
such that $O_- = L_- = K_- = L_0 = K_+$ is a trivial knot, $L_+$ is a positive Hopf link, and $O_+$ is a trivial $2$-component link.  

\begin{figure}[h]
 \centering
 \includegraphics[width=11cm]{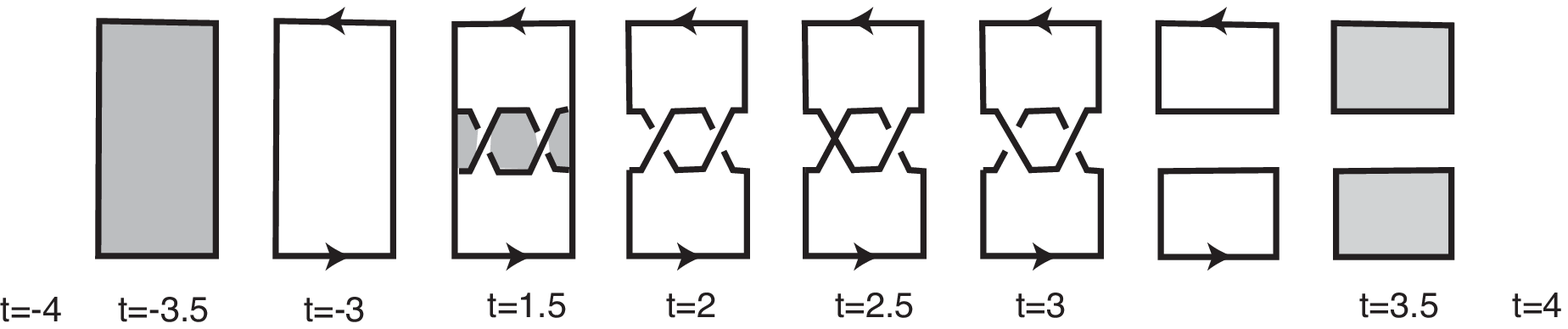}
 \caption{A negative standard singular $2$-knot in a normal form}
 \label{fig:fg_st2knot_negB2}
\end{figure}

(2) When we take $c''_-=1$ and $c'_+ = c'_- = c''_+=0$, we have a surface-knot  in a normal form illustrated 
in Figure~\ref{fig:fg_st2knot_negC2} which is equivalent to $S_-$.  
It is a realizing surface of  a link transformation sequence 
$$ \mathcal{L}: 
\emptyset \to O_- \to L_- \to K_- \to L_0 \to K_+ \to L_+ \to O_+ \to \emptyset $$
such that $O_-$ is a trivial $2$-component link,  
$L_-$ is a negative Hopf link, and $K_-= L_0 = K_+ = L_+ = O_+$ is a trivial knot.  

\begin{figure}[h]
 \centering
 \includegraphics[width=11cm]{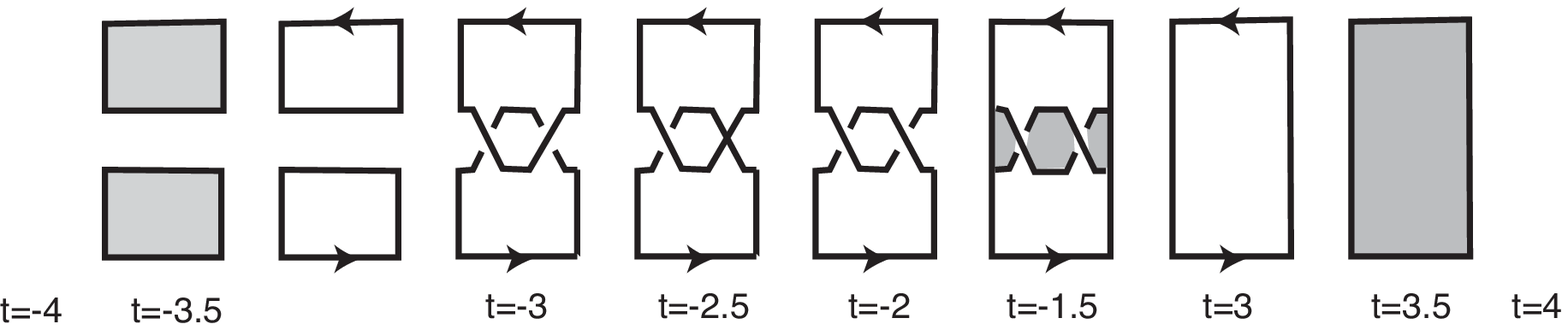}
 \caption{A negative standard singular $2$-knot in another normal form}
 \label{fig:fg_st2knot_negC2}
\end{figure}
}\end{exa}

\begin{exa}{\rm 
Figure~\ref{fig:fg_MT2} shows a Montesinos twin in a symmetric normal form, 
which is a realizing surface of  a link transformation sequence 
$$ \mathcal{L}: 
\emptyset \to O \to L \to K \to L_0 \to K \to L \to O \to \emptyset $$
such that $O$ is a trivial $2$-component link and $L= K= L_0$ is a positive Hopf link.   

\begin{figure}[h]
 \centering
 \includegraphics[width=12.5cm]{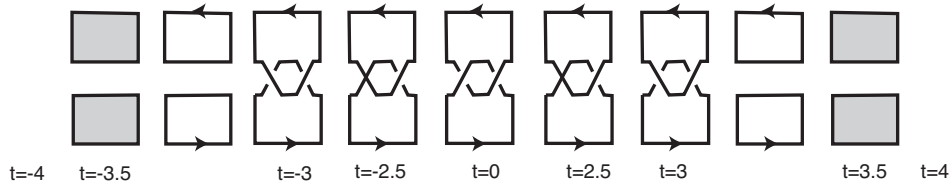}
 \caption{A Montesinos twin in a normal form}
 \label{fig:fg_MT2}
\end{figure}

The surface-link in a symmetric normal form in Figure~\ref{fig:fg_MT2} 
has $c'_-= c''_+=1$ and $c'_+ = c''_-=0$.  
Since a Montesinos twin is equivalent to its mirror image, by considering the mirror image in each cross section in the motion picture of Figure~\ref{fig:fg_MT2},  we have a Montesinos twin in a symmetric normal form with  $c'_-= c''_+=0$ and $c'_+ = c''_-=1$.
}\end{exa}

\section{Proof of Theorems~\ref{thm:AB} and \ref{thm:C}} \label{sect:proofAB}

\begin{proof}[Proof of Theorem~\ref{thm:AB}.]
We show that a ribbon-clasp surface-link $F$  is equivalent to a surface-link in a ribbon-clasp normal form.  
By Theorem~\ref{thm:AA}, $F$ is obtained from an M-trivial $2$-link 
by $1$-handle surgeries.   The M-trivial $2$-link is equivalent to a 
surface-link $F'$ in a ribbon-clasp normal form which is the 
realizing surface of a link transformation sequence 
$$\mathcal{L}': \emptyset \to O' \to L' \to L' \to L'  \to O' \to \emptyset$$ 
 in $\R^3 \times [-3,3]$ with division $-3, -2, -1, 0, 1, 2, 3$
such that $L' \to L' \to L'$ consists of the identity deformations.  
Thus, $F$ is equivalent to a surface-link obtained from $F'$ by $1$-handle surgeries.  Deforming the cores of the $1$-handles up to equivalence, we may assume that 
the cores, say $\gamma_1, \dots, \gamma_m$, are in $\R^3 \times \{0\}$.  
Let $D_1, \dots, D_m$ be bands attaching to $L'$ in $\R^3$ 
whose cores are $\tilde \gamma_1, \dots, \tilde \gamma_m$, where 
$\tilde \gamma_i$ $(i \in \{1, \dots, m\})$ is an arc in $\R^3$ with 
$\tilde \gamma_i \times \{0\} = \gamma_i$.   Put $B_i:= D_i \times [-1/2, 1/2]$, and we have $1$-handles $B_1, \dots, B_m$ attaching to $F'$.  Since  the $1$-handles are equivalent to the original $1$-handles attaching to $F'$, the surface-link $F$ is equivalent to 
the surface-link obtained from $F'$ by $1$-handle surgeries along $B_1, \dots, B_m$, which is in a ribbon-clasp normal form.  

We show that a surface-link $F$ in a ribbon-clasp normal form is a ribbon-clasp surface-link.  Suppose that $F$ is the realizing surface of a link transformation sequence 
$$\mathcal{L}: \emptyset \to O \to L  \to L_0  \to L  \to O \to \emptyset$$ 
 in $\R^3 \times [-3,3]$ with division $-3, -2, -1, 0, 1, 2, 3$ 
 satisfying the conditions of Definition~\ref{dfn:ribbonclaspnormalform}. 
 Let $F'$ be the realizing surface of the link transformation sequence 
$$\mathcal{L}': \emptyset \to O \to L  \to L  \to L  \to O \to \emptyset$$ 
 in $\R^3 \times [-3,3]$ with division $-3, -2, -1, 0, 1, 2, 3$ 
which is obtained from $\mathcal{L}$ by replacing $L  \to L_0  \to L$ 
with the identity deformations $L  \to L  \to L$.  Then $F'$ is an M-trivial $2$-link.   
Note that $F$ is obtained from $F'$ by $1$-handles surgeries.  
(Let $\mathcal{D}= \{D_1, \dots, D_m\}$ be the set of bands attaching to $L$ along which  the $1$-surgery $L \to L_0$ is applied.  Then $D_1 \times  [-1/2, 1/2], \dots, D_m \times [-1/2, 1/2]$ are $1$-handles attaching to $F'$, and $F$ is obtained from $F'$ by surgery along these $1$-handles.)  
By Theorem~\ref{thm:AA}, $F$ is a ribbon-clasp surface-link. 
\end{proof}

\vspace{0.3cm} 

\begin{proof}[Proof of Theorem~\ref{thm:C}.]
First we show that a ribbon-clasp surface-link is equivalent to a surface-link in a symmetric normal form.  Let $F$ be a ribbon-clasp surface-link.  By Theorem~\ref{thm:AB}, we may assume that $F$ is in a ribbon-clasp normal form such that
  it is the realizing surface of a link transformation sequence 
$$\mathcal{L}: \emptyset \to O \to L \to L \to L_0  \to L \to L \to O \to \emptyset$$ 
 in $\R^3 \times [-4,4]$ with division $-4, -3, -2, -1, 0, 1, 2, 3, 4$ 
 satisfying the conditions of Definition~\ref{dfn:ribbonclaspnormalform} with  
the identity deformation  $L\to L$.  
Let $\mathcal{D}$ be the set of bands attaching to $L \to L_0$ along which the $1$-surgery $L \to L_0$ is applied.  Let $F = F_1 \cup \dots \cup  F_\mu$, where $F_i$ is the $i$th component of $F$. The link $L$ is divided into $\mu$ sub-links $L_1, \dots, L_\mu$ such that $L_i$ is the cross-section of $F_i$ at $t=-2$.  Choose a subset $\mathcal{D}''$ of $\mathcal{D}$ such that when we put $K:= {\rm h}(L; \mathcal{D}'')$, 
the $1$-surgery $L \to K$ along $\mathcal{D}''$ is a $1$-surgery of fusion type and the number of components of $K$ is $\mu$.  Let $\mathcal{D}' = \mathcal{D} - \mathcal{D}''$.  
Let 
$$\mathcal{L}': \emptyset \to O \to L \to K \to L_0  \to K \to L \to O \to \emptyset$$ 
be the link transformation sequence obtained from $\mathcal{L}$ by replacing 
$L \to L \to L_0 \to L \to L$ by  $L \to K \to L_0  \to K \to L$ where 
$L \to K$ is the $1$-surgery along $\mathcal{D}''$ and 
$K \to L_0$ is the $1$-surgery along $\mathcal{D}'$.  
By an argument in \cite{KSS1}, we may assume that $K \to L_0$ is a $1$-surgery of fission type and the number of components of $L_0$ is $\mu + g$. 
Then $F$ is equivalent to the realizing surface of $\mathcal{D}''$  
 in $\R^3 \times [-4,4]$ with division $-4, -3, -2, -1, 0, 1, 2, 3, 4$, which is in a symmetric normal form.  
 
We show that  a surface-link in a symmetric normal form is a ribbon-clasp surface-link.  Let $F$ be a surface-link, with $\mu$ components and $g$ total genus, in a symmetric normal form, i.e., it is the realizing surface of a link transformation sequence  
$$\mathcal{L}: \emptyset \to O \to L \to K \to L_0  \to K \to L \to O \to \emptyset$$ 
in $\R^3 \times [-4,4]$ with division $-4, -3, -2, -1, 0, 1, 2, 3, 4$ satisfying the following conditions. 
\begin{itemize} 
\item[(1)] $O$ is a trivial link and $O \to \emptyset$ is a $2$-surgery. 
\item[(2)] $L$ is an H-trivial link and $L \to O$ is a Hopf-splitting deformation. 
\item[(3)] $K$ is a link with $\mu$ components and  $K \to L$ is a $1$-surgery of fission type. 
\item[(4)] $L_0$ is a link with $\mu + g$ components and $L_0 \to K$ is a $1$-surgery of fusion type. 
\item[(5)] Let $\mathcal{D}'$ (or $\mathcal{D}''$) be the set of bands along which the $1$-surgery $L_0 \to K$ (or $K \to L$) is applied. Then any band of $\mathcal{D}''$ is disjoint from any band of $\mathcal{D}'$. 
\item[(6)] The realizing surface is symmetric with respect to $\R^3 \times \{0\}$. 
\end{itemize} 
Put $\mathcal{D}:= \mathcal{D}' \cup \mathcal{D}''$, which is a set of bands attaching to $L_0$ such that the surgery result is $L$.  
Let 
$$\mathcal{L}': \emptyset \to O \to L \to L \to L_0  \to L \to L \to O \to \emptyset$$ 
be the sequence obtained from $\mathcal{L}$ by replacing 
$L \to K \to L_0  \to K \to L$ with $L \to L \to L_0  \to L \to L$ where 
$L\to L$ is the identity deformation, $L_0 \to L$ is the $1$-surgery along $\mathcal{D}$ and $L \to L_0$ is its reverse.  Then $F$ is equivalent to the realizing surface of $\mathcal{L}'$ in $\R^3 \times [-4,4]$ with division $-4, -3, -2, -1, 0, 1, 2, 3, 4$.  Finally, the surface-link is equivalent to a surface-link in a ribbon-clasp normal form.  
\end{proof}

\section{Proof of Theorem~\ref{thm:B}}

First we prepare lemmas.  

Let $L \to L'$ be a simple crossing change deformation and let 
$F(L \to L')_{[a,b]}$ be the  realizing surface in $\R^3 \times [a,b]$.   
Let $W$ be a small $3$-disk in $\R^3$ such that $W \times [a,b] \cap  F(L \to L')_{[a,b]} = \emptyset$. 

\begin{lem} \label{lem:doubleHopfA} 
In the above situation, there is a link transformation sequence 
$\mathcal{L}: L=L_0 \to L_1 \to L_2 \to L_3 \to L_4= L'$ satisfying the following. 
\begin{itemize} 
\item[(1)] $F(L \to L')_{[a,b]}$ is ambient isotopic to 
$F(\mathcal{L})_{[a,b]}$ by an isotopy of $\R^3 \times [a,b]$ rel 
$\R^3 \times \{a, b\}$.  
\item[(2)] $L \to L_1$ is a $1$-surgery along two bands, and $L_1$ is the split union of a link $\tilde{L'}$ and a Hopf link, where $\tilde{L'}$ is a link equivalent to $L'$. 
\item[(3)] $L_1 \to L_2$ is an isotopic deformation moving $\tilde{L'}$ onto $L'$  and moving the Hopf link into $W$ so that $L_2$ is the split union of $L'$ and a Hopf link in $W$, 
\item[(4)] $L_2 \to L_3$ is a simple Hopf-splitting deformation that changes the Hopf link in $W$ to a trivial $2$-component link in $W$, 
\item[(5)] $L_3 \to L_4$ is a $2$-surgery that eliminates the trivial $2$-component link  in $W$. 
\end{itemize}
\end{lem} 

\begin{proof}
Let $U$ be a $3$-disk in $\R^3$ which is a regular neighborhood of the double point of  the singular link appearing as the cross-section of $F(L \to L')_{[a,b]}$ 
at $t = (a+b)/2$.  The motion picture in the top of Figure~\ref{fig:hopf1} shows the 
restriction of $F(L \to L')_{[a,b]}$ to $U \times [a,b]$ in a case of a negative double point.   
Replace it with the surface as shown in the second and the third lows of Figure~\ref{fig:hopf1}, namely, which is the realizing surface of a sequence 
$\mathcal{L}^{(0)}: L=L^{(0)}_0 \to L^{(0)}_1 \to L^{(0)}_2 \to L^{(0)}_3 \to L^{(0)}_4= L'$ satisfying the following. 
\begin{itemize} 
\item[(1)] $F(L \to L')_{[a,b]}$ is ambient isotopic to 
$F(\mathcal{L}^{(0)})_{[a,b]}$ by an isotopy of $\R^3 \times [a,b]$ rel 
$\R^3 \times \{a, b\}$.  
\item[(2)] $L \to L^{(0)}_1$ is a $1$-surgery along two bands, and $L^{(0)}_1$ is the split union of a link $\tilde{L'}$ and a Hopf link $H^{(0)}$, where $\tilde{L'}$ is equivalent to $L'$. 
\item[(3)] $L^{(0)}_1 \to L^{(0)}_2$ is a simple crossing change deformation between the components of $H^{(0)}$. Let $O^{(0)}_2$ be the trivial  $2$-component link obtained from $H^{(0)}$.  
So $L^{(0)}_2$ is a split union of $\tilde{L'}$ and $O^{(0)}_2$. 
\item[(4)] $L^{(0)}_2 \to L^{(0)}_3=\tilde{L'}$ is a $2$-surgery that eliminates $O^{(0)}_2$.  
\item[(5)] $L^{(0)}_3 \to L^{(0)}_4$ is an isotopy deformation. 
\end{itemize}
We change the sequence $\mathcal{L}^{(0)}$  by inserting an isotopic deformation 
just after $L \to L^{(0)}_1$ 
such that a small $3$-disk in $\R^3$ containing the Hopf link $H^{(0)}$ is moved into $W$.  Then we obtain a desired sequence $\mathcal{L}$. 
\end{proof}

%
\begin{figure}[h]
 \centering
 \includegraphics[width=8.5cm]{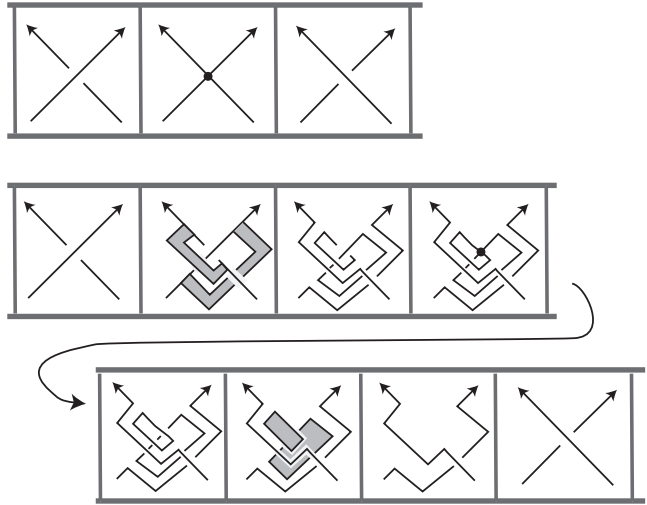}
 \caption{Moving a double point}
 \label{fig:hopf1}
\end{figure}

\vspace{0.3cm} 

Let $\mathcal{L}: L_0 \to \dots \to L_m$ be a link transformation sequence and let 
$F(\mathcal{L})_{[a,b]}$ be the realizing surface in $\R^3 \times [a,b]$.  
Suppose that for some $n$, $L_n \to L_{n+1}$ in $\mathcal{L}$ is a simple crossing  change deformation.  Let $W$ be a small $3$-disk in $\R^3$ such that $W \times [a,b] \cap  F(\mathcal{L})_{[a,b]} = \emptyset$. 

\begin{lem} \label{lem:doubleHopfB} 
In the above situation, there is a link transformation sequence 
$\mathcal{L}^{(1)}: L^{(1)}_0 \to \dots \to L^{(1)}_{m+3}$ satisfying the following. 
\begin{itemize} 
\item[(1)] 
$F(\mathcal{L})_{[a,b]}$ is ambient isotopic to 
$F(\mathcal{L}^{(1)})_{[a,b]}$ by an isotopy of $\R^3 \times [a,b]$ rel 
$\R^3 \times \{a, b\}$.  
\item[(2)] 
For each $k \in \{1, \dots, n\}$,  
$L^{(1)}_{k-1} \to L^{(1)}_{k}$ is the same with $L_{k-1} \to L_{k}$.  
\item[(3)] 
$L^{(1)}_{n} \to L^{(1)}_{n+1}$ is a $1$-surgery along two bands, and $L^{(1)}_{n+1}$ is the split union of a link $\tilde{L}_{n+1}$ and a Hopf link, where $\tilde{L}_{n+1}$ is a link equivalent to $L_{n+1}$. 
\item[(4)] 
$L^{(1)}_{n+1} \to L^{(1)}_{n+2}$ is an isotopic deformation moving $\tilde{L}_{n+1}$ onto $L_{n+1}$  and moving the Hopf link into $W$ so that $L^{(1)}_{n+2}$ is the split union of $L_{n+1}$ and a Hopf link $H$ in $W$.  
\item[(5)] 
For each $k \in \{n+3, \dots, m+1\}$, 
$L^{(1)}_{k}$ is the split union of $L_{k-1}$ and $H$, and the link transformation  
$L^{(1)}_{k-1} \to L^{(1)}_{k}$ is the same with 
 $L_{k-2} \to L_{k-1}$ on $L_{k-2}$ and the identity on $H$.  
\item[(6)] 
$L^{(1)}_{m+1} \to L^{(1)}_{m+2}$ is a simple Hopf-splitting deformation that changes the Hopf link $H$ to a trivial $2$-component link in $W$. 
\item[(7)] 
$L^{(1)}_{m+2} \to L^{(1)}_{m+3}$ is a $2$-surgery that eliminates the trivial $2$-component link in $W$. Then $L^{(1)}_{m+3} = L_m$. 
\end{itemize}
\end{lem} 

\begin{proof}
Applying  Lemma~\ref{lem:doubleHopfA} to $L_n \to L_{n+1}$, we see that 
there is a link transformation sequence 
$\mathcal{L}': L'_0 \to \dots \to L'_{m+3}$ satisfying the following. 
\begin{itemize} 
\item[(1)] 
$F(\mathcal{L})_{[a,b]}$ is ambient isotopic to 
$F(\mathcal{L}')_{[a,b]}$ by an isotopy of $\R^3 \times [a,b]$ rel 
$\R^3 \times \{a, b\}$.  
\item[(2)] 
For each $k \in \{1, \dots, n\}$, 
$L'_{k-1} \to L'_{k}$ is the same with $L_{k-1} \to L_{k}$.  
\item[(3)] 
$L'_{n} \to L'_{n+1}$ is a $1$-surgery along two bands, and $L'_{n+1}$ is the split union of a link $\tilde{L}_{n+1}$ and a Hopf link, where $\tilde{L}_{n+1}$ is a link equivalent to $L_{n+1}$. 
\item[(4)] 
$L'_{n+1} \to L'_{n+2}$ is an isotopic deformation moving $\tilde{L}_{n+1}$ onto $L_{n+1}$  and moving the Hopf link into $W$ so that $L'_{n+2}$ is the split union of $L_{n+1}$ and a Hopf link $H$ in $W$.   
\item[(5)] 
$L'_{n+2} \to L'_{n+3}$ is a simple Hopf-splitting deformation that changes the Hopf link $H$ in $W$ to a trivial $2$-component link.  
\item[(6)] 
$L'_{n+3} \to L'_{n+4}$ is a $2$-surgery that eliminates the trivial $2$-component  link in $W$. Then $L'_{n+4} = L_{n+1}$. 
\item[(7)] 
For $k \in \{n+5, \dots, m+3\}$, 
$L'_{k-1} \to L'_{k}$ is the same with $L_{k-4} \to L_{k-3}$.  

\end{itemize}
By an ambient isotopy, we deform the realizing surface $F(\mathcal{L}')_{[a,b]}$ 
so that the simple Hopf-splitting deformation and the $2$-surgery in the above (5) and (6) occur after the link transformations in (7).  Then we have a desired link transformation sequence $\mathcal{L}^{(1)}$. 
\end{proof}

\vspace{0.3cm}

\begin{proof}[Proof of Theorem~\ref{thm:B}.]
Let $F$ be a surface-link.  
Divide the set of positive (or negative) double  points of $F$ into two groups, say $G'_+$ and $G''_+$ (or $G'_-$ and $G''_-$), with $\# G'_+ =  c'_+$ and $\# G''_+ =  c''_+$ (or $\# G'_- =  c'_-$ and $\# G''_- =  c''_-$).  

By a similar argument with \cite{KSS1}, there is a link transformation sequence $\mathcal{L}^{(0)}$ such that a realizing surface of 
$\mathcal{L}^{(0)}$  is equivalent to $F$,  
each crossing change deformation is simple, 
there is no $0$-surgery in  $\mathcal{L}^{(0)}$ except the first link transformation, and there is no $2$-surgery in  $\mathcal{L}^{(0)}$ except the last link transformation.  
 (Note that the first link transformation  
of $\mathcal{L}^{(0)}$  is a $0$-surgery from the empty link to a trivial link, and the last link transformation is a $2$-surgery from a trivial link to the empty link.) 

Take a double point $p_1$ from $G'_+ \cup G'_-$ and let 
$L^{(0)}_{n_0} \to L^{(0)}_{n_0+1}$ be the 
simple crossing change deformation in $\mathcal{L}^{(0)}$ whose realizing surface  
contains $p_1$.  Apply Lemma~\ref{lem:doubleHopfB} to $\mathcal{L}^{(0)}$ with $L^{(0)}_{n_0} \to L^{(0)}_{n_0+1}$  and we have a link transformation sequence 
such that a realizing surface of the sequence is equivalent to $F$, 
each crossing change deformation is simple, and the last three link transformations are a simple Hopf-splitting deformation from an H-trivial link to a trivial link, 
a $2$-surgery that eliminates the trivial $2$-component link obtained by the simple Hopf-splitting deformation, and a $2$-surgery from a trivial link to the empty link.  
Combine  the last two link transformations into a link transformation which is a $2$-surgery from a trivial link to the empty link.  Let $\mathcal{L}^{(1)}$ be the link transformation sequence. Then 
$\mathcal{L}^{(1)}$ is a link transformation sequence such that a realizing surface is equivalent to $F$,  each crossing change deformation is simple, 
and the last two link transformations are a simple Hopf-splitting deformation from an H-trivial link to a trivial link and a $2$-surgery from the trivial link to the empty link. 

Take a double point $p_2$ from $(G'_+ \cup G'_-) - \{p_1\}$ and let 
$L^{(1)}_{n_1} \to L^{(1)}_{n_1+1}$ be the 
simple crossing change deformation in $\mathcal{L}^{(1)}$ whose realizing surface  
contains $p_2$.   By the same argument as above, we have a link transformation sequence $\mathcal{L}^{(2)}$ 
such that a realizing surface of $\mathcal{L}^{(2)}$  
 is equivalent to $F$,  each crossing change deformation is simple, 
and the last three  link transformations are 
a simple Hopf-splitting deformation from an H-trivial link to 
an H-trivial link, a  simple Hopf-splitting deformation from the $H$-trivial link to 
a trivial link 
and a $2$-surgery from the trivial link to the empty link. 
 
Continue this procedure until we cover all double points belonging to $G'_+ \cup G'_-$, and we have   
a link transformation sequence 
such that a realizing surface of the sequence 
 is equivalent to $F$,  each crossing change deformation is simple, 
and the last $c'_+ + c'_- +1$ link transformations are 
$c'_+ + c'_-$ simple Hopf-splitting deformations from H-trivial links,  
and a $2$-surgery from a trivial link to the empty link. 

For this link transformation sequence, we apply a similar argument on double points  belonging to $G''_+ \cup G''_-$. 
However, this time we move the double points forward in the time direction.  
Then we have a  link transformation sequence 
such that a realizing surface of the sequence 
 is equivalent to $F$,  each crossing change deformation is simple, 
 there is no $0$-surgery in the sequence except the first link transformation, 
  there is no $2$-surgery in the sequence except the last link transformation, 
the first  $c''_+ + c''_- +1$ link transformations are a 
$0$-surgery from the empty link to a trivial link and 
$c''_+ + c''_-$  inverses of simple Hopf-splitting deformations, and the 
last $c'_+ + c'_- +1$ link transformations are 
$c'_+ + c'_-$ simple Hopf-splitting deformations and a $2$-surgery from a trivial link to the empty link. 

Combine the first $c''_+ + c''_-$ inverses of  simple Hopf-splitting deformations to a 
Hopf-splitting deformation, and combine the last $c'_+ + c'_-$  simple Hopf-splitting deformations 
to a Hopf-splitting deformation.  
Then we have a  link transformation sequence, say 
$\mathcal{L}' : \emptyset \to O'_-  \to L'_- \to 
\dots \to L'_+ \to O'_+ \to \emptyset$, where $L'_- \to 
 \dots \to L'_+$ is a sequence consisting of $1$-surgeries, such that 
 a realizing surface of the sequence  $\mathcal{L}'$ is equivalent to $F$.  
 Applying the argument in \cite{KSS1} to this sequence with $L'_- \to 
 \dots \to L'_+$, we obtain a sequence 
 $\mathcal{L} : \emptyset \to O_-  \to L_- \to K_- \to L_0 \to K_+ \to 
L_+ \to O_+ \to \emptyset$ as desired.    
\end{proof}




\end{document}